\documentclass[11pt,oneside]{amsart}

\usepackage{bm}

\usepackage{fix-cm}
\DeclareMathAlphabet{\mathcal}{OMS}{cmsy}{m}{n}

\usepackage{tikz-cd}

\usepackage{tcolorbox}

\usepackage[arrow,curve,matrix,tips,2cell]{xy}

\usepackage{fancyhdr}

\usepackage{anyfontsize}
\usepackage{mathrsfs}

\usepackage{accents}

\usepackage{bbm}

\usepackage{amsmath}
\usepackage{amscd}
\usepackage{amssymb}
\usepackage{latexsym}
\usepackage{url}

\usepackage{enumitem}

\usepackage{graphicx}

\newtheorem{theorem}{Theorem}[section]
\newtheorem*{theorem*}{Theorem}
\newtheorem{lemma}[theorem]{Lemma}
\newtheorem*{lemma*}{Lemma}
\newtheorem{corollary}[theorem]{Corollary}
\newtheorem{proposition}[theorem]{Proposition}

\newtheorem{remark}[theorem]{Remark}
\newtheorem{definition}[theorem]{Definition}
\newtheorem*{definition*}{Definition}

\newtheorem{question}[theorem]{Question}
\newtheorem*{question*}{Question}

\newtheorem{example}[theorem]{Example}
\newtheorem{examples}[theorem]{Examples}

\newtheorem{introthm}{Theorem}

\newtheorem{introcor}[introthm]{Corollary}

\makeatletter
\def\revddots{\mathinner{\mkern1mu\raise\p@
\vbox{\kern7\p@\hbox{.}}\mkern2mu
\raise4\p@\hbox{.}\mkern2mu\raise7\p@\hbox{.}\mkern1mu}}
\makeatother 
\newcommand{\bgl}{\begin{equation}} 
\newcommand{\egl}{\end{equation}}
\newcommand{\bgloz}{\begin{equation*}} 
\newcommand{\egloz}{\end{equation*}}
\newcommand{\bgln}{\begin{eqnarray}} 
\newcommand{\egln}{\end{eqnarray}}
\newcommand{\bglnoz}{\begin{eqnarray*}} 
\newcommand{\eglnoz}{\end{eqnarray*}}
\newcommand{\btheo}{\begin{theorem}}
\newcommand{\etheo}{\end{theorem}}
\newcommand{\btheooz}{\begin{theorem*}}
\newcommand{\etheooz}{\end{theorem*}}

\newcommand{\blemma}{\begin{lemma}}
\newcommand{\elemma}{\end{lemma}}
\newcommand{\blemmaoz}{\begin{lemma*}}
\newcommand{\elemmaoz}{\end{lemma*}}
\newcommand{\bproof}{\begin{proof}}
\newcommand{\eproof}{\end{proof}}
\newcommand{\bbew}{\begin{beweis}}
\newcommand{\ebew}{\end{beweis}}
\newcommand{\bremark}{\begin{remark}\em}
\newcommand{\eremark}{\end{remark}}
\newcommand{\bdefin}{\begin{definition}}
\newcommand{\edefin}{\end{definition}}
\newcommand{\bdefinoz}{\begin{definition*}}
\newcommand{\edefinoz}{\end{definition*}}
\newcommand{\bex}{\begin{example}}
\newcommand{\eex}{\end{example}}
\newcommand{\bexs}{\begin{examples}}
\newcommand{\eexs}{\end{examples}}

\newcommand{\bprop}{\begin{proposition}}
\newcommand{\eprop}{\end{proposition}}
\newcommand{\bcor}{\begin{corollary}}
\newcommand{\ecor}{\end{corollary}}
\newcommand{\bfa}{\begin{cases}} 
\newcommand{\efa}{\end{cases}}

\newcommand{\bquestion}{\begin{question}}
\newcommand{\equestion}{\end{question}}
\newcommand{\bquestionoz}{\begin{question*}}
\newcommand{\equestionoz}{\end{question*}}
\newcommand{\cA}{\mathcal A}

\newcommand{\cO}{\mathcal O}

\newcommand{\cY}{\mathcal Y}
\newcommand{\cZ}{\mathcal Z}

\def\Nz{\mathbb{N}}

\def\Tz{\mathbb{T}}
\def\Zz{\mathbb{Z}}

\def\1z{\mathbb{1}}
\newcommand{\an}[1]{``#1''} 
\newcommand{\ti}{\tilde}

\newcommand{\lori}{\longrightarrow}

\newcommand{\ma}{\mapsto} 
\newcommand{\onto}{\twoheadrightarrow} 
\newcommand{\into}{\hookrightarrow} 

\def\SEMI{\mbox{$\times\kern-2pt\vrule height5pt width.6pt \kern3pt $}}

\renewcommand{\dim}{{\rm dim}\,}
\newcommand{\rk}{{\rm rk}\,}

\newcommand{\ev}{\operatorname{ev}} 
\newcommand{\abs}[1]{\lvert#1\rvert} 
\newcommand{\norm}[1]{\left\|#1\right\|} 
\newcommand{\defeq}{\mathrel{:=}} 

\newcommand{\dop}{\text{: }} 
\newcommand{\ilim}{\varinjlim} 
\newcommand{\plim}{\varprojlim} 
\newcommand{\lge}{\left\{} 
\newcommand{\rge}{\right\}} 
\newcommand{\gekl}[1]{\lge #1 \rge} 
\newcommand{\menge}[2]{\gekl{ #1 \dop #2 }} 
\newcommand{\isom}{\xrightarrow{\raisebox{-1ex}[0ex][0ex]{$\sim$}}} 

\newcommand{\oset}[2]{%
  \mathop{#2}\limits^{\vbox to -1.66ex{%
  \kern -1.4ex\hbox{$#1$}\vss}}}

\newcommand{\pars}{\setlength{\parindent}{0cm} \setlength{\parskip}{0.5cm}}
\newcommand{\pari}{\setlength{\parindent}{0.5cm} \setlength{\parskip}{0cm}}
\newcommand{\nopar}{\setlength{\parindent}{0cm} \setlength{\parskip}{0cm}}

\usepackage{newtxtext}
\usepackage{newtxmath}

\begin{document}

\title{Constructing C*-diagonals in AH-algebras}

\thispagestyle{fancy}

\author{Xin Li}
\author{Ali I. Raad}

\address{Xin Li, School of Mathematics and Statistics, University of Glasgow, University Place, Glasgow G12 8QQ, United Kingdom}
\email{Xin.Li@glasgow.ac.uk}
\address{Ali I. Raad, Department of Mathematics, KU Leuven, 200B Celestijnenlaan, 3001 Leuven,
Belgium}
\email{ali.imadraad@kuleuven.be}


\thanks{This project has received funding from the European Research Council (ERC) under the European Union's Horizon 2020 research
and innovation programme (grant agreement No. 817597).
The second author was supported by the Internal KU Leuven BOF project C14/19/088 and project G085020N funded by the Research Foundation Flanders (FWO)}

\begin{abstract}
We construct Cartan subalgebras and hence groupoid models for classes of AH-algebras. Our results cover all AH-algebras whose building blocks have base spaces of dimension at most one as well as Villadsen algebras, and thus go beyond classifiable simple C*-algebras.
\end{abstract}

\maketitle


\setlength{\parindent}{0cm} \setlength{\parskip}{0.5cm}

\section{Introduction}

Many important classes of C*-algebras can be described naturally as groupoid C*-algebras, for instance C*-algebras attached to group actions (crossed products) or semigroup actions on topological spaces (Cuntz-Krieger algebras or graph algebras) or AF-algebras. Kumjian and Renault \cite{Kum,Ren} developed the notions of C*-diagonals and Cartan subalgebras which provide a general framework for constructing groupoid models for C*-algebras. Based on these ideas, it was shown in \cite{Li18} that every classifiable simple C*-algebra is a (twisted) groupoid C*-algebra. The term \an{classifiable} refers to C*-algebras which have been classified in the context of the Elliott classification programme, which has been initiated by the work of Glimm, Dixmier, Bratteli and Elliott, and has seen tremendous progress in recent years (see \cite{KP, Phi, GLNa, GLNb, EGLN, TWW,EGLN17a,EGLN17b,GLI,GLII,GLIII} and the references therein).

The goal of the present paper is to construct Cartan subalgebras and hence (twisted) groupoid models for classes of C*-algebras which go beyond the classifiable simple case. The motivation for finding Cartan subalgebras is that they automatically produce underlying twisted groupoids such that the ambient C*-algebras can be written as the corresponding twisted groupoid C*-algebras (by the main results in \cite{Kum,Ren}). In a broader context, Cartan subalgebras in C*-algebras are attracting attention because of close connections to continuous orbit equivalence for topological dynamical systems, interactions with geometric group theory \cite{Li16,Li17,Li_DQH} as well as links to the UCT question \cite{BL16,BL17}. With this in mind, we expect that the groupoids we construct will be of independent interest, for instance from the point of view of topological dynamical systems.

We focus on approximately homogeneous C*-algebras (AH-algebras), i.e., C*-algebras which arise as inductive limits of homogeneous ones. For our first goal of constructing twisted groupoid models for general (i.e., not necessarily simple) AH-algebras, the notion of maximally homogeneous connecting maps (introduced in \cite{Tho_JOT92} in a special case) turns out to be crucial. Roughly speaking, a homomorphism between homogeneous C*-algebras is maximally homogeneous if its image has maximal dimension in each fibre. Such connecting maps allow us to apply the general machinery for constructing Cartan subalgebras and hence twisted groupoid models for inductive limit C*-algebras from \cite{BL17,Li18}. Then the question of which connecting maps can be approximated by maximally homogeneous ones becomes interesting, because such connecting maps may be replaced by maximally homogeneous ones without changing the inductive limit C*-algebra up to isomorphism. It turns out that this approximation problem has come up in the classification of AH-algebras \cite{Li97,Gon} and is related to questions in topology.

Our second goal is to construct Cartan subalgebras in classes of C*-algebras which lie outside of the scope of the Elliott classification programme for C*-algebras. This includes for example Villadsen algebras of the first and second kind \cite{Vil98,Vil99} as well as Tom's examples of non-classifiable C*-algebras \cite{Toms}. The general framework is given by AH-algebras with generalized diagonal connecting maps, special cases of which appear in \cite{Niu} and \cite{Goo}. Apart from covering the examples above, our construction also produces Cartan subalgebras in other C*-algebras, for instance Goodearl algebras \cite{Goo} (see also \cite[Example~3.1.7]{Ror}). At this point, we would like to mention that Giol and Kerr constructed non-classifiable simple crossed products of the form $C(X) \rtimes \Zz$ (see \cite{GK}), which by construction also contain canonical Cartan subalgebras (namely $C(X)$). 

The interest behind constructing Cartan subalgebras in these non-classifiable C*-algebras stems from the observation that either the underlying groupoids are not almost finite in the sense of \cite{Mat,Kerr,MW} or we obtain an interesting example of a groupoid together with a twist such that the untwisted groupoid C*-algebra is classifiable while the twisted groupoid C*-algebra is not. Therefore our construction produces interesting new examples of twisted groupoids which can be viewed as the underlying dynamical structures giving rise to non-classifiable C*-algebras. In this context, an interesting and natural task would be to characterize in dynamical terms those twisted groupoids or Cartan pairs whose C*-algebras are classifiable. We hope that the twisted groupoids constructed in the present paper will be of interest for this task by providing interesting test cases.

Let us now present the main results of this paper.
\nopar

\begin{introthm}[see Corollary~\ref{cor:DiagMaxHom}]
Unital AH-algebras with unital, injective and maximally homogeneous connecting maps have C*-diagonals.
\end{introthm}
We refer to Definition~\ref{def:MaxHom} for the precise definition of maximally homogeneous homomorphisms between homogeneous C*-algebras. In combination with the result (see Theorem~\ref{thm:dim=1}, which is based on \cite{Li97}) that for homogeneous C*-algebras with base spaces of dimension one, general connecting maps can be approximated by maximally homogeneous ones, we obtain the following consequence:

\begin{introthm}[see Theorem~\ref{thm:dim<=1}]
\label{introthm:dim<=1}
Unital AH-algebras whose building blocks have base spaces of dimension at most one have C*-diagonals.
\end{introthm}
Note that additional work is required to make sure that injective connecting maps can be approximated by injective, maximally homogeneous connecting maps (see Theorem~\ref{thm:dim=1}).
\pars

In particular, Theorem~\ref{introthm:dim<=1} covers all AI- and all A$\Tz$-algebras. In combination with \cite{GJLP}, we derive the following application:
\begin{introcor}[see Corollary~\ref{cor:AHbdddim}]
Every unital AH-algebra with bounded dimension, the ideal property, and torsion-free $K$-theory has a C*-diagonal.
\end{introcor}

Let us now turn to a second class of AH-algebras. The precise definition of \an{generalized diagonal connecting maps} can be found in Definition~\ref{def:GenDiag}.
\begin{introthm}[see Theorem~\ref{thm:DiagDiagConn}]
Unital AH-algebras with unital, injective, generalized diagonal connecting maps have C*-diagonals.
\end{introthm}
\nopar

This allows us to cover several classes of C*-algebras, including non-classifiable ones.
\begin{introcor}[see Corollary~\ref{cor:DiagDiagConn_Ex}]
The following classes of examples have C*-diagonals:
\begin{itemize}
\item Villadsen algebras of the first kind \cite{Vil98},
\item Tom's examples of non-classifiable C*-algebras \cite{Toms},
\item Goodearl algebras \cite{Goo} (see also \cite[Example~3.1.7]{Ror}), 
\item AH-algebras models for dynamical systems in \cite[Example~2.5]{Niu}, 
\item Villadsen algebras of the second kind \cite{Vil99}.
\end{itemize}
\end{introcor}
\pars

The last class of examples, Villadsen algebras of the second type, leads to the following interesting phenomenon:
\begin{introcor}[see Corollary~\ref{cor:sr}]
\label{introcor:sr}
For each $m = 2, 3, \dotsc$ or $m = \infty$, there exists a twisted groupoid $(\bar{G},\bar{\Sigma})$ such that $C^*_r(\bar{G},\bar{\Sigma})$ is a unital, separable, simple C*-algebra of stable rank $m$, whereas $C^*_r(\bar{G})$ is a unital, separable, simple C*-algebra of stable rank one. 
\end{introcor}
\nopar

Stable rank is a notion of dimension in noncommutative topology. The property of stable rank one has important consequences regarding the structure of C*-algebras. Corollary~\ref{introcor:sr} shows that stable rank depends not only on the underlying groupoid model for a C*-algebra, but also on the twist.
\pars

This paper is partly based on contents of the PhD thesis of the second author, completed at Queen Mary University of London and the University of Glasgow.

\section{Preliminaries}

\subsection{AH-algebras}

An approximately homogeneous (AH) C*-algebra is of the form $A = \ilim_n \gekl{A_n, \varphi_n}$, with $A_n = \bigoplus_i A_n^i$ for all $n \in \Nz$, where $i$ runs through a finite index set which depends on $n$, and $A_n^i$ is a homogeneous C*-algebra over the compact, connected, Hausdorff space $Z_n^i$. In other words, $A_n^i$ is the C*-algebras of continuous sections of a locally trivial C*-bundle $\cA_n^i$ over $Z_n^i$ with fibre $M_{r_n^i}$. 

We will focus on homogeneous C*-algebras of the following form: for each $n$ and $i$, let $p_n^i \in C(Z_n^i) \otimes M_{r_n^i}$ be projections of rank $r_n^i$, and then our building block is given by $A_n \defeq \bigoplus_i p_n^i ( C(Z_n^i) \otimes M_{r_n^i} ) p_n^i$, and $A = \ilim_n \gekl{A_n, \varphi_n}$.

In the above, $\varphi_n$ are homomorphisms $A_n \to A_{n+1}$. Since we will focus on unital AH-algebras in this paper, we assume throughout that the connecting maps are unital.
\pari

Moreover, for our construction of C*-diagonals, it is important that the connecting maps are injective. This can always be arranged without changing the inductive limit C*-algebra up to isomorphism (see \cite[Theorem~2.1]{EGL05}).

For a criterion for simplicity of AH-algebras, we refer the reader to \cite[Proposition~3.1.2]{Ror}.

Here is a sufficient condition, which we will use, for when two inductive limit C*-algebras are isomorphic (see \cite[\S~2.3]{Ror} for a more general criterion): 
\nopar

\begin{lemma}\label{lemma: sufficient iso limit}
Let $\varinjlim \{A_n,\phi_n\}$ and $\varinjlim \{A_n,\psi_n\}$ be inductive limit C*-algebras. For $n, k \in \mathbb{N}$ with $k <n$ let $\phi_{n,k}=\phi_{n-1} \circ \phi_{n-2} \circ \ldots \circ \phi_k$ (and define $\psi_{n,k}$ similarly). Let $\mathcal{F}_n \subset A_n$ be a finite set generating $A_n$ such that $\mathcal{F}_n$ contains $$\big(\bigcup\limits_{k=1}^{n-1} \phi_{n,k}(\mathcal{F}_k) \big) \cup \big( \bigcup\limits_{k=1}^{n-1} \psi_{n,k}(\mathcal{F}_k) \big),$$ and assume that $$\norm{\phi_n(a)-\psi_n(a)}< 2^{-n},$$ for all $a \in \mathcal{F}_n$, $n \in \mathbb{N}$. Then $\varinjlim (A_n,\phi_n) \cong \varinjlim (A_n,\psi_n)$.
\end{lemma}
\begin{proof}
We obtain an approximate intertwining in the sense of Definition 2.3.1 in \cite{Ror} or Definition 2 in \cite{klaus3}. Indeed, in the latter definition just choose the vertical maps to be the identity and the diagonal maps to be the $\psi_n$'s. The result follows by Theorem 3 in \cite{klaus3}, or Proposition 2.3.2 in \cite{Ror}.
\end{proof}
\pars

\subsection{Cartan subalgebras in inductive limit C*-algebras}
\label{ss:CartanLimC}

\bdefin[\cite{Kum,Ren}, in the unital case]
\label{def:CartanSubalgebra}
A sub-C*-algebra $B$ of a unital C*-algebra $A$ is called a Cartan subalgebra if
\nopar

\begin{itemize}
\item $B$ is maximal abelian;
\item $B$ is regular, i.e., $N_A(B) \defeq \menge{n \in A}{n B n^* \subseteq B \ \text{and} \ n^* B n \subseteq B}$ generates $A$ as a C*-algebra;
\item there exists a faithful conditional expectation $P: \: A \onto B$.
\end{itemize}

A pair $(A,B)$, where $B$ is a Cartan subalgebra of a C*-algebra $A$, is called a Cartan pair.
\pari

A Cartan subalgebra $B$ is called a C*-diagonal if $(A,B)$ has the unique extension property, i.e., every pure state of $B$ extends uniquely to a (necessarily pure) state of $A$.
\edefin
\nopar

Given two Cartan pairs $(A_1,B_1)$ and $(A_2,B_2)$, we write $(A_1,B_1) \cong (A_2,B_2)$ if there exists an isomorphism $\varphi: \: A_1 \isom A_2$ with $\varphi(B_1) = B_2$.
\pars

\btheo[\cite{Kum,Ren} and \cite{KM,Raa}]
\label{def:KumjianRenault}
For every Cartan pair $(A,B)$, there exists a (up to isomorphism) unique twisted groupoid $(G,\Sigma)$ such that $G$ is {\'e}tale, Hausdorff, locally compact and effective, such that $(A,B) \cong (C^*_r(G,\Sigma),C_0(G^{(0)}))$. In addition, $B$ is a C*-diagonal if and only if $G$ is principal.
\etheo

In the setting we are interested in, the two notions of Cartan subalgebras and C*-diagonals coincide.
\bprop[{\cite[Proposition~6.1]{Ren}}]
\label{prop:Cartan=Diag}
If $A$ is a homogeneous C*-algebra, then every Cartan subalgebra of $A$ is a C*-diagonal.
\eprop

Recall the results in \cite{BL17} and \cite[\S~5]{Li18} implying when an inductive limit of Cartan subalgebras $B \defeq \ilim_n \gekl{B_n,\varphi_n}$ is a Cartan subalgebra of the inductive limit $A = \ilim_n \gekl{A_n,\varphi_n}$. 
\nopar

\btheo
\label{thm:LimCartan}
Let $(A_n,B_n)$ be Cartan pairs with normalizers $N_n \defeq N_{A_n}(B_n)$ and faithful conditional expectations $P_n: \: A_n \onto B_n$. Let $\varphi_n: \: A_n \to A_{n+1}$ be injective homomorphisms with $\varphi_n(B_n) \subseteq B_{n+1}$, $\varphi_n(N_n) \subseteq N_{n+1}$ and $P_{n+1} \circ \varphi_n = \varphi_n \circ P_n$ for all $n$. Then $\ilim \gekl{B_n,\varphi_n}$ is a Cartan subalgebra of $\ilim \gekl{A_n,\varphi_n}$.
\pari

If all $B_n$'s are C*-diagonals, then $\ilim \gekl{B_n,\varphi_n}$ is a C*-diagonal of $\ilim \gekl{A_n,\varphi_n}$.
\etheo
\pars

The following explains our assumptions on the connecting maps. As above, let $(A_n,B_n)$ be Cartan pairs with normalizers $N_n \defeq N_{A_n}(B_n)$ and faithful conditional expectations $P_n: \: A_n \onto B_n$. Let $\varphi_n: \: A_n \to A_{n+1}$ be injective homomorphisms. Let $(G_n,\Sigma_n)$ be the twisted groupoid models for $(A_n,B_n)$.
\bprop
\label{prop:HomGPDmodel}
The following are equivalent:
\nopar

\begin{enumerate}
\item[(i)] $\varphi_n(B_n) \subseteq B_{n+1}$, $\varphi(N_{n}) \subseteq N_{n+1}$, $P_{n+1} \circ \varphi_n = \varphi_{n} \circ P_n$;
\item[(ii)] There exists a twisted groupoid $(H_n,T_n)$, with $H_n$ \'{e}tale and effective, and twisted groupoid homomorphisms $(i_n,\imath_n): \: (H_n,T_n) \to (G_{n+1},\Sigma_{n+1})$, $(\dot{p}_n,p_n): \: (H_n,T_n) \to (G_n,\Sigma_n)$, where $i_n$ is an embedding with open image and $\dot{p}_n$ is surjective, proper and fibrewise bijective, such that $\varphi_n = (\imath_n)_* \circ (p_n)^*$.
\end{enumerate}
\eprop
\pars

Let us recall the construction of the twisted groupoid model $(\bar{G},\bar{\Sigma})$ for the pair $(A,B)$. Following \cite[\S~5]{Li18}, define $(G_{n,0}, \Sigma_{n,0}) \defeq (G_n,\Sigma_n)$, $(G_{n,m+1}, \Sigma_{n,m+1}) \defeq p_{n+m}^{-1}(G_{n,m},\Sigma_{n,m}) \subseteq (H_{n+m},T_{n+m})$ for all $n$ and $m = 0, 1, \dotsc$, $(\bar{G}_n,\bar{\Sigma}_n) \defeq \plim_m \gekl{(G_{n,m},\Sigma_{n,m}), p_{n+m}}$ for all $n$. The inclusions $(H_n,T_n) \into (G_{n+1},\Sigma_{n+1})$ induce embeddings with open image $\bar{i}_n: \: (\bar{G}_n,\bar{\Sigma}_n) \into (\bar{G}_{n+1},\bar{\Sigma}_{n+1})$, allowing us to define $(\bar{G},\bar{\Sigma}) \defeq \ilim \gekl{(\bar{G}_n,\bar{\Sigma}_n), \bar{i}_n}$. As explained in \cite[\S~5]{Li18}, $(\bar{G},\bar{\Sigma})$ is a groupoid model for $(A,B)$ in the sense that we have a canonical isomorphism $A \isom C^*_r(\bar{G},\bar{\Sigma})$ sending $B$ to $C_0(\bar{G}^{(0)})$.

\section{Cartan subalgebras in AH-algebras with maximally homogeneous connecting maps}

In the following, we introduce and study the notion of \an{maximally homogeneous} maps, and explain the connection to existing notions.

As above, let $A_n = \bigoplus_i A_n^i$ be a homogeneous C*-algebra, which we view as the C*-algebras of continuous sections of a locally trivial C*-bundle $\cA_n^i$ over $Z_n^i$ with fibre $M_{r_n^i}$. Set $Z_n \defeq \coprod_i Z_n^i$ and define $r(x) \defeq r_n^i$ for $x \in Z_n^i$.

Let $\varphi: \: A_n \to A_{n+1}$ be a unital homomorphism. We know that for every $y \in Z_{n+1}$, $\ev_y \circ \varphi \sim_u \bigoplus_{\mu \in M_y} \ev_{x_{\mu}}$, where $M_y$ is an index set and $x_{\mu}$ are some points in $Z_n$. This is the case because, up to unitary equivalence, all irreducible representations of $A_n$ are of the form $\ev_x$ for some $x \in Z_n$.
\bdefin
\label{def:MaxHom}
We call $\varphi$ maximally homogeneous if for every $y \in Z_{n+1}$, the points $x_{\mu}$, $\mu \in M_y$, are pairwise distinct.
\edefin
\nopar

Note that because $\varphi$ is unital, we must have $r(y) = \sum_{\mu \in M_y} r(x_{\mu})$.
\pars

\bremark
This definition coincides with the original one from \cite{Tho_JOT92} (see also \cite{Raa_PhD}) up to restricting to direct summands. Let us explain this:
\pari

Suppose that $A_n = \bigoplus_i C(Z_n^i) \otimes M_{r_n^i}$ and $A_{n+1} = \bigoplus_j C(Z_{n+1}^j) \otimes M_{r_{n+1}^j}$, where $Z_n^i$ and $Z_{n+1}^j$ are connected. Let $e_n^i$ be a rank one projection in $M_{r_n^i}$. Then $\rk \varphi(1 \otimes e_n^i)(y)$ is constant on $Z_{n+1}^j$; let $m_i^j$ be this value. Then we know that $\ev_y \circ \varphi \sim_u \bigoplus_i \bigoplus_{\mu \in M_y^i} \ev_{x_{\mu}}$, where $M_y^i$ is an index set with $\# M_y^i = m_i^j$ and $x_{\mu}$ are some points in $Z_n$. Then
$$
 \dim(\ev_y \circ \varphi(A_n)) = \sum_i (\# \menge{x_{\mu}}{\mu \in M_y^i}) \cdot (r_n^i)^2 \leq \sum_i m_i^j (r_n^i)^2,
$$
and equality holds if and only if $\# \menge{x_{\mu}}{\mu \in M_y^i}) = m_i^j$ for all $i$, i.e., the $x_{\mu}$ (for all possible $\mu$ and $i$) are pairwise distinct.
\eremark
\pars

The following clarifies the behaviour of maximally homogeneous maps with regard to C*-diagonals (in our setting of homogeneous C*-algebras, all Cartan subalgebras are C*-diagonals by Proposition~\ref{prop:Cartan=Diag}).
\nopar

\bprop
\label{prop:MaxHomDiag}
Let $B_n \subseteq A_n$ be a C*-diagonal and $\varphi: \: A_n \to A_{n+1}$ be a maximally homogeneous homomorphism. Then the following hold:
\begin{enumerate}
\item[(i)] There exists a unique C*-diagonal $B_{n+1}$ in $A_{n+1}$ such that $\varphi(B_n) \subseteq B_{n+1}$, and it is given by $B_{n+1} \defeq C^*(C(Z_{n+1}), \varphi(B_n))$.
\item[(ii)] $\varphi(N_{A_n}(B_n)) \subseteq N_{A_{n+1}}(B_{n+1})$, where $N_{\bullet}(\cdot)$ stands for the set of normalizers.
\item[(iii)] $P_{n+1} \circ \varphi = \varphi \circ P_n$, where $P_n$ and $P_{n+1}$ are the conditional expectations $A_n \onto B_n$ and $A_{n+1} \onto B_{n+1}$.
\end{enumerate}
\eprop
\bproof
(i) Given $y \in Z_{n+1}$, let $M_y$ and $x_{\mu}$, $\mu \in M_y$, be as before. For $d = 1, \dotsc, r(x_{\mu})$, let $b_{\mu}^d \in B_n$ be such that $b_{\mu}^d(x_{\mu})$, $d = 1, \dotsc, r(x_{\mu})$, are pairwise orthogonal projections (necessarily of rank one) and $b_{\nu}^d(x_{\mu}) = 0$ if $\nu \neq \mu$. Enumerate the $b_{\mu}^d$, i.e., write $\gekl{b_l} = \gekl{b_{\mu}^d}_{\mu,d}$. It then follows that $\varphi(b_l)(y)$ are $r(y)$ pairwise orthogonal projections (necessarily of rank one). Thus $B_{n+1}(y) \cong D_{r(y)}$, where $D_r$ is the canonical C*-diagonal in $M_r$. Moreover, we can choose $b_{\mu}^d \in B_n$ such that $b_{\mu}^d(x_{\mu})$, $d = 1, \dotsc, r(x_{\mu})$, are pairwise orthogonal projections and $b_{\nu}^d(x_{\mu}) = 0$ if $\nu \neq \mu$ for all $y'$ in a small neighbourhood around $y$. So $B_{n+1} \vert_{Z_{n+1}^j}$ is the C*-algebra of continuous sections of a locally trivial sub-C*-bundle of $\cA_{n+1}^j$ over $Z_{n+1}^j$ with fibre $D_{r_{n+1}^j}$. Hence $B_{n+1}$ is a C*-diagonal by \cite[\S~2]{LR}.
\pars

(ii) This now follows because given $a \in N_{A_n}(B_n)$, $\varphi(a)$ normalizes $\varphi(B_n)$ by definition and commutes with $C(Z_{n+1})$.

(iii) For $y \in Z_{n+1}$, let $\gekl{b_l}$ be as in the proof of (i) and let $\mu_l \in M_y$ be the unique index such that $b_l(x_{\mu_l}) \neq 0$. Then, for $a \in A_n$, 
$$
 P_n(a)(x_{\mu}) = \sum_{l, \, \mu_l = \mu} b_l(x_{\mu}) a(x_{\mu}) b_l(x_{\mu}) = \sum_l b_l(x_{\mu}) a(x_{\mu}) b_l(x_{\mu}).
$$
We conclude that
\begin{align*}
 P_{n+1}(\varphi(a))(y) &= \sum_l \varphi(b_l)(y) \varphi(a)(y) \varphi(b_l)(y) \, \sim_u \, \sum_l (b_l(x_{\mu}) a(x_{\mu}) b_l(x_{\mu}))_{\mu}\\
 &= \big( \sum_l b_l(x_{\mu}) a(x_{\mu}) b_l(x_{\mu}) \big)_{\mu} = (P_n(a)(x_{\mu}))_{\mu} \, \sim_u \, \varphi(P_n(a))(y).
\end{align*}
Here $\sim_u$ refers to a fixed unitary equivalence $\ev_y \circ \varphi \sim_u \bigoplus_{\mu} \ev_{x_{\mu}}$.
\eproof
\pars

\bcor
\label{cor:DiagMaxHom}
Suppose that $A_n$ are homogeneous C*-algebras as above and $\varphi_n: \: A_n \to A_{n+1}$ are unital, injective and maximally homogeneous homomorphisms. If $A_1$ has a C*-diagonal, then $A = \ilim_n \gekl{A_n, \varphi_n}$ has a C*-diagonal.
\ecor
\nopar

\bproof
This follows from Proposition~\ref{prop:MaxHomDiag} and Theorem~\ref{thm:LimCartan}.
\eproof
\pars

\btheo
\label{thm:dim=1}
Suppose that $A = \ilim_n \gekl{A_n,\phi_n}$, where $A_n = \bigoplus_i C(Z_n^i) \otimes M_{r_n^i}$ and $Z_n^i$ are non-degenerate $1$-dimensional connected CW-complexes (i.e., non-degenerate finite graphs), and $\phi_n$ are unital, injective homomorphisms. Then there exist unital, injective and maximally homogeneous homomorphisms $\varphi_n: \: A_n \to A_{n+1}$ such that we still have $A \cong \ilim_n \gekl{A_n, \varphi_n}$.
\etheo
\nopar

\begin{proof}
To simplify notation slightly, consider an injective and unital homomorphism $$\phi: \bigoplus\limits_{j=1}^N C(Z_j) \otimes M_{n_j} \rightarrow \bigoplus\limits_{i=1}^M C(W_i) \otimes M_{m_i},$$ where the $Z_j$'s and $W_i$'s are 1-dimensional connected CW-complexes. For each $i \in \{1,\ldots,M\}$  we associate to $W_i$ a finite connected tree $\mathcal{D}_i$ and a surjective map $p_i: \mathcal{D}_i \rightarrow W_i$ inducing an injective homomorphism $\iota_i:  C(W_i) \otimes M_{m_i} \rightarrow  C(\mathcal{D}_i) \otimes M_{m_i}, f \to f\circ p_i$. One may obtain $\mathcal{D}_i$ by taking a suitable connected compact subset of the universal covering tree associated to $W_i$. 
\pars

Throughout the rest of the proof, fix $j \in \{1,\ldots,N\}$. Identify the edges of $Z_j$ with $[0,1]$. Let $0 < \delta < \frac{1}{2}$ to be chosen, with the property that 
\begin{equation}\label{greater than delta}
  \delta < (\sum\limits_{i=1}^M m_i)^{-1}.  
\end{equation} Divide every edge of $Z_j$ into finitely many equal-sized strips of length at most $\delta$. For a strip identified with $[a,b]$ consider the function supported on this strip with value $\frac{2(x-a)}{b-a}$ for $x \in [a,\frac{a+b}{2}]$, and with value $\frac{2(b-x)}{b-a}$ for $x \in [\frac{a+b}{2},b]$. For adjacent strips $[a, b]$, $[b, c]$ (where we also consider strips that share a vertex as a boundary point as adjacent), let $m_1$ be the midpoint of $[a,b]$ and $m_2$ the midpoint of $[b,c]$. Define a function supported on $[m_1,b] \cup [b, m_2]$ taking value $\frac{(x-m_1)}{b-m_1}$ for $x \in [m_1,b]$, and taking value $\frac{m_2-x}{m_2-b}$ for $x \in [b, m_2].$ Let $H_j$ be a finite subset of central elements in $C(Z_j) \otimes M_{n_j}$ consisting of all such functions on all our strips. For any finite set $F_j$ containing the finitely many generators of $C(Z_j) \otimes M_{n_j}$, we let $G_j=F_j \cup H_j$. For any $z \in Z_j$ there exists an element $h \otimes 1 \in H_j$ with $h(z) \ge \frac{1}{2}$.

For $i \in \{1, \ldots, M\}$, let $P_i= \Pi_i \circ \phi \circ \eta_j(1 \otimes 1_{n_j})$, where $\eta_j$ denotes the canonical inclusion of the $j^{\text{th}}$ summand into a direct sum, and $\Pi_i$ denotes the canonical projection of the direct sum onto its $i^{\text{th}}$ summand. Assume $P_i$ has trace $K_i$ at every evaluation. Let
\begin{equation*}
    \phi_i: C(Z_j) \otimes M_{n_j} \rightarrow P_i(C(W_i) \otimes M_{m_i})P_i, \; \; f \to \Pi_i(\phi(\eta_j(f))),
\end{equation*}
\begin{equation*}
    \mathrm{Ad}(\nu_i): P_i(C(W_i) \otimes M_{m_i})P_i \xrightarrow{\cong} C(W_i) \otimes M_{K_i},
\end{equation*}
where $\mathrm{Ad}(\nu_i)$ is obtained by noting that $[P_i]_0=[1 \otimes 1_{K_i}]_0$ in the $K_0$ group and that $C(Z_j)$ has cancellation as it has stable rank one (see \cite[\S~1.1]{Ror}), implying that $P_i \sim 1 \otimes 1_{K_i}$. Apply Corollary 2.1.8 in \cite{Li97} to obtain a maximally homogeneous homomorphism 
\begin{equation*}
    \chi_i:  C(Z_j) \otimes M_{n_j} \rightarrow C(W_i) \otimes M_{K_i} 
\end{equation*}such that 
\begin{equation}\label{delta 2}
    \norm{\chi_i(a)-\mathrm{Ad}(\nu_i) \circ \phi_i (a)} < \delta
\end{equation} for all $a \in G_j$. 

By Theorem 3 in \cite{klaus4} we may write
\begin{equation}\label{diag form 3}
   \iota_i \circ \chi_i(f)(t)=u_i(t)(\sum\limits_{s=1}^{\frac{K_i}{n_j}} f(\lambda^i_s(t)) \otimes q_s)u_i(t)^*
\end{equation} for all $f \in C(Z_j) \otimes M_{n_j}$ and $t \in \mathcal{D}_i$, and where $u_i$ is a unitary in $C(\mathcal{D}_i) \otimes M_{K_i}$ and where the $q_s$ are the canonical minimal projections of $M_{\frac{K_i}{n_j}}$ when identifying $M_{K_i}$ with $M_{n_j} \otimes M_{\frac{K_i}{n_j}}$. Here the functions $\lambda^i_1, \ldots, \lambda^i_{\frac{K_i}{n_j}} : \mathcal{D}_i \rightarrow Z_j$ are called eigenvalue functions. Note that the homomorphism 
\begin{equation*}
    \alpha : C(Z_j) \otimes M_{n_j} \rightarrow \bigoplus\limits_{i=1}^M C(\mathcal{D}_i) \otimes M_{K_i}, \; \; f \to (\iota_i \circ \mathrm{Ad}(\nu_i)\circ \phi_i (f))_i
\end{equation*}
is injective because $\phi$ is. Using this we conclude that given any $z \in Z_j$, and $h\in H_j$ such that $h(z) \ge \frac{1}{2}$, we have that 
\begin{equation*}
    h(z) \in \mathrm{sp}_{C(Z_j) \otimes M_{n_j}}(h \otimes 1)=\mathrm{sp}_{\bigoplus\limits_{i=1}^M C(\mathcal{D}_i) \otimes M_{K_i}}(\alpha(h \otimes 1))
\end{equation*} and so there must exist some $i \in \{1,\ldots,M\}$ such that \begin{equation*}
    h(z) \in \mathrm{sp}_{C(\mathcal{D}_{i}) \otimes M_{m_{i}}}(\Pi_{i}(\alpha(h \otimes 1)))
\end{equation*}
from which it follows that there exists some $t \in \mathcal{D}_{i}$ such that 
\begin{equation*}
  h(z) \in  \mathrm{sp}_{M_{m_{i}}}(\iota_{i} \circ \mathrm{Ad}(\nu_{i}) \circ \phi_{i}(h\otimes 1)(t))
\end{equation*}
From \eqref{delta 2} it follows that 
\begin{equation*}
   \norm{\iota_{i}\circ \chi_{i}(h \otimes 1)(t) -  \iota_{i}\circ \mathrm{Ad}(\nu_{i}) \circ \phi_{i}(h \otimes 1)(t)} < \delta
\end{equation*}
from which it follows by Lemma 1.1 in \cite{klaus2} and \eqref{diag form 3} that there exists some $s \in \{1,\ldots,\frac{K_i}{n_j}\}$ such that $$\abs{h(\lambda^i_s(t))-h(z)} \le \norm{\chi_i(h \otimes 1) - \mathrm{Ad}(\nu_i) \circ \phi_i(h \otimes 1)} < \delta < \frac{1}{2}.$$ From this it follows that $h(\lambda^i_s(t)) \neq 0$ and so both $\lambda^i_s(t)$ and $z$ belong to the support of $h$, meaning that $\lambda^i_s(t)$ is distance at most $\delta$ away from $z$. This argument has shown that points in $Z_j \setminus \bigcup\limits_{i=1}^M\bigcup\limits_{s=1}^{\frac{K_i}{n_j}} \lambda^i_s(\mathcal{D}_i)$ are at most a distance $\delta$ from $\bigcup\limits_{i=1}^M\bigcup\limits_{s=1}^{\frac{K_i}{n_j}} \lambda^i_s(\mathcal{D}_i)$. 

We may assume $Z_j \setminus \bigcup\limits_{i=1}^M\bigcup\limits_{s=1}^{\frac{K_i}{n_j}} \lambda^i_s(\mathcal{D}_i)$ is a union of finitely many pairwise disjoint connected open sets which, by the previous arguments, can be made arbitrarily small depending on $\delta$. That the union is finite follows from the fact that $\bigcup\limits_{i=1}^M\bigcup\limits_{s=1}^{\frac{K_i}{n_j}} \lambda^i_s(\mathcal{D}_i)$ contains finitely many connected components, as each $\lambda^i_s(\mathcal{D}_i)$ is connected. Each such connected open set is then homeomorphic to an open interval or an open asterisk (but in the latter case this can also be written as a union of pairwise disjoint connected open sets each homeomorphic to an open interval). Hence we identify each open set with $(a_r,b_r)$ where $a_r$ and $b_r$ are identified with points in $Z_j$ (not necessarily belonging to the same edge). Hence we identify $Z_j \setminus \bigcup\limits_{i=1}^M\bigcup\limits_{s=1}^{\frac{K_i}{n_j}} \lambda^i_s(\mathcal{D}_i)$ with the disjoint union $\bigcupdot\limits_{r=1}^R (a_r,b_r)$ for some $R \in \mathbb{N}$.
 
 Let $\mathcal{U}=\{(a_r,b_r) : 1 \le r \le R\}.$ We will call $\mathcal{C} \subset \mathcal{U}$ a \emph{chain} if the closure of the union of its elements is connected. We will call $\mathcal{M} \subset \mathcal{U}$ a \emph{maximal chain} if it is a chain and is not contained in any bigger chain. Given any maximal chain $\mathcal{M}$, relabel its elements as $\mathcal{M}=\{(a_1,b_1), \ldots, (a_T, b_T)\}$. Note that the closure of the union of the elements of $\mathcal{M}$, call this $C$, is contractible. Indeed if it were not, $C$ would have to contain an edge $e$ of $Z_j$, meaning that $T$ must be greater than $\delta^{-1}$, which is greater than $\sum\limits_{i=1}^Mm_i$ by \eqref{greater than delta}. Hence there will be more than $\sum\limits_{i=1}^Mm_i$ isolated points in $Z_j$ which are the constant images of eigenvalue functions, but we do not have more than $\sum\limits_{i=1}^Mm_i$ eigenvalue functions for this to work, yielding a contradiction. Hence $C$ is homeomorphic to either a line or an asterisk.
 
 Let us focus on the first case when $C$ is homeomorphic to a line. We may without loss of generality assume that $b_1=a_2$, $b_2=a_3$, $\ldots$, $b_{T-1}=a_T.$ As a first sub-case assume all the $a_r$'s and $b_r$'s are not vertex points in $Z_j$ (up to identification). There exists some $i \in \{1,\ldots,M\}$ and $s \in \{1,\ldots,\frac{K_i}{n_j}\}$ and an eigenvalue function $\lambda^i_s$ whose image contains $a_1$, and some $t \in \mathcal{D}_i$ such that $\lambda^i_s(t) \in (a_1-\rho,a_1]$ for some $\rho >0$ which can be made arbitrarily small, depending on $\delta.$ By continuity of the eigenvalue functions, we may without loss of generality assume that $t$ is not a vertex point in $\mathcal{D}_i$. After identifying edges of $\mathcal{D}_i$ with $[0,1]$, we may choose an open interval $U$ inside an edge of $\mathcal{D}_i$, containing $t$, small enough such that $\lambda^i_s\vert_U$ still has image in $(a_1-\rho,a_1]$, and such that the image of $\lambda^i_s\vert_U$ does not meet the image of any other $\lambda^i_{s^\prime}\vert_U$. The latter assumption can be made because maximal homogeneity of \eqref{diag form 3} implies that the elements of $\{\lambda^i_s(t) : 1 \le s \le \frac{K_i}{n_j}\}$ are pairwise distinct, and so by continuity the images of these eigenvalue functions when restricted to a small open set around $t$ will still not meet. For every $t_k \in p_i^{-1}(p_i(t))$, it follows from \eqref{diag form 3} that $$\{f\circ \lambda^i_1(t_k), \ldots, f \circ \lambda^i_\frac{K_i}{n_j}(t_k)\}=\{f\circ \lambda^i_1(t), \ldots, f \circ \lambda^i_\frac{K_i}{n_j}(t)\},$$ for all $f \in C(Z_j) \otimes M_{n_j}$. Since $C(Z_j)$ separates points it follows that \begin{equation}\label{above}
     \{ \lambda^i_1(t_k), \ldots,  \lambda^i_\frac{K_i}{n_j}(t_k)\}=\{ \lambda^i_1(t), \ldots, \lambda^i_\frac{K_i}{n_j}(t)\},
 \end{equation} and so there exists a permutation $\mu_k \in \Sigma_{\frac{K_i}{n_j}}$ such that 
 \begin{equation}\label{permutation}
     \lambda^i_s(t)=\lambda^i_{\mu_k(s)}(t_k)
 \end{equation} for all $s \in \{1,\ldots, \frac{K_i}{n_j}\}$. Hence for each $k$ there exists an open interval $V_k$ around $t_k$ and a homeomorphism $h_k: V_k \rightarrow U_k$ where $U_k$ is some open interval around $t$, with $h_k(t_k)=t$ and such that 
 \begin{equation*}
     \lambda^i_s\vert_{U_k}=\lambda^i_{\mu_k(s)}\circ h_k^{-1}\vert_{U_k}
 \end{equation*} for all $s \in \{1,\ldots, \frac{K_i}{n_j}\}$. Indeed, since the elements in \eqref{above} are pairwise distinct, points in a small neighbourhood around $t$ will correspond to points in a small neighbourhood around $t_k$, where by a correspondence we mean that they have the same image under $p_i$. By choosing small enough neighbourhoods and using continuity of the eigenvalue functions we obtain our claim. 
 
 By taking an intersection of all the $U_k$'s (which are finitely many as $\mathcal{D}_i$ is finite), and assuming $U$ is a subset of this intersection, we obtain open sets $V_k$ around $t_k$ and homeomorphisms $h_k: V_k \rightarrow U$ with $h_k(t_k)=t$ and such that 
 \begin{equation}\label{homeo eigen}
     \lambda^i_s\vert_{U}=\lambda^i_{\mu_k(s)}\circ h_k^{-1}\vert_{U}.
 \end{equation} Perturb $\lambda^i_s$ on $U$ to a map $\omega^i_s$ which contains the image of $\lambda^i_s$ but also $[a_1,b_1]$. For any connected closed set $V$ inside $U$ and any point $t^\prime$ in $V$ we may further assume that $\omega^i_s$ takes value $b_1$ only at the point $t^\prime$, agreeing with $\lambda^i_s$ outside $V$ and being greater than or equal to $\lambda^i_s$ on $V$ (where we are the using the canonical order on intervals). Use \eqref{homeo eigen} to perturb $\lambda^i_{\mu_k(s)}$ correspondingly, so that it achieves value $b_1$ at the unique point $h_k^{-1}(t^\prime)$ inside a closed subset of $V_k$, for all $k$. 
 
Now since $b_1=a_2 \not\in Z_j \setminus \bigcup\limits_{i=1}^M\bigcup\limits_{s=1}^{\frac{K_i}{n_j}} \lambda^i_s(\mathcal{D}_i)$ there exists $\overline{i} \in \{1,\ldots,M\}$ and $\overline{s} \in \{1,\ldots, \frac{K_i}{n_j}\}$ and an eigenvalue function $\lambda^{\overline{i}}_{\overline{s}}$ achieving the value $b_1=a_2$. In fact this eigenvalue function must be constant, as it cannot take any value in $(a_2,b_2)$, and by assumption $a_2$ is not a vertex point. Exactly as before, perturb this to a function $\omega^{\overline{i}}_{\overline{s}}$ with image $[a_2,b_2]$ and we may assume that the value of $b_2$ is achieved only at $t^\prime$ if $\overline{i}=i$. Also perturb the corresponding eigenvalue functions (as in the sense of \eqref{homeo eigen}). Repeat this procedure for the points $a_3,a_4,\ldots,a_T$. This concludes the first sub-case.
 
 For the second sub-case, assume some $a_t=b_{t-1}$ is a vertex in $Z_j$, for some $t \in \{2,\ldots,T\}$. In such a case, divide the maximal chain $\mathcal{M}$ into two chains $\{(a_1,b_1), \ldots, (a_{t-1},b_{t-1})\}$ and $\{(a_t,b_t), \ldots, (a_{T},b_{T})\}$, and perform the same procedure of the first sub-case for each chain, making sure that the eigenvalue functions perturbed to achieve the point $a_t$ do so at different points in $\bigsqcup\limits_i \mathcal{D}_i$.
 
 Let us now turn to the second case, where $C$ is homeomorphic to an asterisk, with vertex $v$. As a first sub-case, we assume that $v \in (a_r,b_r)$ for some $r \in \{1,\ldots,T\}$. Divide the maximal chain into chains, one containing $(a_r,b_r)$, and the others contained in edges coming out of $v$. For these ones, we may perturb as in the second sub-case of the first case, making sure the eigenvalue functions achieving the value $v$ do so at different points. For the chain containing $(a_r,b_r)$, we perturb as in the first sub-case of the first case, making sure that the point at which $v$ is achieved differs from where it is achieved by the perturbed eigenvalue functions on the remaining chains. 
 
 As a second sub-case, we assume that $v$ is not contained in any of the $(a_r,b_r)$'s. From this it follows that $v \in \bigcup\limits_{i=1}^M\bigcup\limits_{s=1}^{\frac{K_i}{n_j}} \lambda^i_s(\mathcal{D}_i)$, and so there is some eigenvalue function $\lambda$ achieving the value $v$. Break up the maximal chain into sub-chains each contained in a unique edge coming out of $v$. For all but one of these sub-chains, use the second sub-case of the first case to perturb the eigenvalue functions, making sure that where $v$ is achieved is at distinct points. Now perturb $\lambda$ to cover one of the elements of the final remaining sub-chain, done in a way such that the value $v$ is not achieved on any point in which the previous perturbations achieve this value.
 
 The procedure we have described ensures that the perturbed eigenvalue functions do not agree pointwise at points where they achieve the values of $C \setminus \bigcup\limits_{m_t \in \mathcal{M}} m_t$. However, there may be a finite number of points, in the boundary of $C$, where the perturbed eigenvalue functions agree pointwise with some non-perturbed eigenvalue functions. In such a case it is straightforward to perturb the originally non-perturbed eigenvalue functions to achieve these problematic points at a slightly different value in $\bigsqcup\limits_i \mathcal{D}_i$ than their original. Since the number of such problematic points are finite, this is achievable.
 
 Let $\{\omega^i_s : 1 \le i \le M, 1 \le s \le \frac{K_i}{n_j} \}$ denote the set of perturbed eigenvalue functions (where we set $\omega^i_s= \lambda^i_s$ if $\lambda^i_s$ was not perturbed). We have the following list of properties of our perturbed eigenvalue functions:
\nopar

 \begin{enumerate}[label=(\alph*)]
     \item\label{oneone} The union $\bigcup\limits_{i=1}^M\bigcup\limits_{s=1}^{\frac{K_i}{n_j}} \omega^i_s(\mathcal{D}_i)$ equals $Z_j$. This is because the images of the perturbed functions cover all the elements of $\mathcal{U}$ as well as the images of the original unperturbed eigenvalue functions.
     \item\label{twotwo} We have that $\norm{\omega^i_s(t)-\lambda^i_s(t)}$ can be made as small as we like (depending on $\delta$), uniformly over $t \in \mathcal{D}_i$. This is because when perturbing an eigenvalue function to cover $(a_r,b_r) \in \mathcal{U}$ the maximal difference between the perturbed function and the original is $b_r-(a_r-\rho)$ where $\rho$ was chosen small enough depending on $\delta$, and $b_r-a_r$ is smaller than $\delta$ by construction.
     \item\label{threethree} For each $i \in \{1,\ldots,M\}$, the elements of $\{\omega^i_s(t) : 1 \le s\le \frac{K_i}{n_j}\}$ are pairwise distinct. This is due to the final argument made above. 
     \item\label{fourfour} For $t \in \mathcal{D}_i$ on which $\lambda^i_s$ witnessed a perturbation, and for any $t_k \in \mathcal{D}_i$ such that $p_i(t_k)=p_i(t)$, we have that $\omega^i_s(t)=\omega^i_{\mu_k(s)}(t_k)$ where the permutations $\mu_k$ were introduced in \eqref{permutation}. This is due to \eqref{homeo eigen} and the arguments following it. 
 \end{enumerate}
\pars
 
 For $i \in \{1,\ldots,M\},$ define 
\begin{equation*}
    \overline{\gamma}_i: C(Z_j) \otimes M_{n_j} \rightarrow C(\mathcal{D}_i) \otimes M_{K_i}, \;\; \overline{\gamma}_i(f)(t)=u_i(t)(\sum\limits_{s=1}^{\frac{K_i}{n_j}} f(\omega^i_s(t)) \otimes q_s)u_i(t)^*
\end{equation*}
for all $t \in \mathcal{D}_i$, and where the $u_i$'s and $q_s$'s are as in \eqref{diag form 3}. Define
\begin{equation*}
    \overline{\gamma}: C(Z_j) \otimes M_{n_j} \rightarrow \bigoplus\limits_{i=1}^M C(\mathcal{D}_i) \otimes M_{K_i}, \; \; f \to (\overline{\gamma}_1(f),\ldots, \overline{\gamma}_M(f))
\end{equation*}
Now define \begin{equation*}
    \gamma_i: C(Z_j) \otimes M_{n_j} \rightarrow C(W_i) \otimes M_{K_i}, \; \; \gamma_i(f)(w) = \overline{\gamma}_i(f)(t_w), \;\; t_w \in p_i^{-1}(w).
\end{equation*}
To see that this is well-defined let $t_1$ and $t_2$ belong to $\mathcal{D}_i$ such that $p_i(t_1)=p_i(t_2)$. Note then that, as in \eqref{above} and \eqref{permutation} above, we have that there exists a permutation $\mu$ such that $\lambda^i_s(t_1)=\lambda^i_{\mu(s)}(t_2)$ for all $s \in \{1,\ldots, \frac{K_i}{n_j}\}$. Fix an $s_0 \in \{1,\ldots, \frac{K_i}{n_j}\}$. Then choose a function $f$ in \eqref{diag form 3} which is 1 on $\lambda^i_{s_0}(t_1)$ and 0 on all the other $\lambda^i_s(t_1)$'s. By using \eqref{diag form 3} and evaluating on $t_1$ and $t_2$ (which yields the same result) we obtain that
\begin{equation}\label{c1}
    u_i(t_1)(1 \otimes q_{s_0})u_i(t_1)^*=u_i(t_2)(1 \otimes q_{\mu(s_0)})u_i(t_2)^*.
\end{equation}
By choosing a function $f$ in \eqref{diag form 3} which takes constant value $c \in M_{n_j}$, and evaluating \eqref{diag form 3} at $t_1$ and $t_2$ we obtain that 
\begin{equation}\label{c2}
     u_i(t_1)(c \otimes 1)u_i(t_1)^*=u_i(t_2)(c \otimes 1)u_i(t_2)^*.
\end{equation}
This implies that for all $f\in C(Z_j) \otimes M_{n_j}$ we have
\begin{equation*}
    \begin{split}
        \overline{\gamma}_i(f)(t_1) &= u_i(t_1)(\sum\limits_{s=1}^{\frac{K_i}{n_j}} f(\omega^i_s(t_1)) \otimes q_s)u_i(t_1)^* \\ &= \sum\limits_{s=1}^{\frac{K_i}{n_j}}(u_i(t_1)(f(\omega^i_s(t_1)) \otimes 1)u_i(t_1)^*)(u_i(t_1)(1 \otimes q_s)u_i(t_1)^*) \\ & = \sum\limits_{s=1}^{\frac{K_i}{n_j}}(u_i(t_2)(f(\omega^i_s(t_1)) \otimes 1)u_i(t_2)^*)(u_i(t_2)(1 \otimes q_{\mu(s)})u_i(t_2)^*) \\ &= u_i(t_2)(\sum\limits_{s=1}^{\frac{K_i}{n_j}}f(\omega^i_s(t_1)) \otimes q_{\mu(s)})u_i(t_2)^* \\ & = u_i(t_2)(\sum\limits_{s=1}^{\frac{K_i}{n_j}}f(\omega^i_{\mu(s)}(t_2)) \otimes q_{\mu(s)})u_i(t_2)^* \\ & =  \overline{\gamma}_i(f)(t_2),
    \end{split}
\end{equation*}
where the third equality is due to \eqref{c1} and \eqref{c2}, and the penultimate equality is due to property \ref{fourfour} above. Hence $\gamma_i$ is well-defined. 

Now define
\begin{equation*}
    \gamma: C(Z_j) \otimes M_{n_j} \rightarrow \bigoplus\limits_{i=1}^M C(W_i) \otimes M_{K_i}, \; \; f \to (\gamma_1(f),\ldots,\gamma_M(f)).
\end{equation*}
Define 
\begin{equation*}
    \varphi_j: C(Z_j) \otimes M_{n_j} \rightarrow \bigoplus\limits_{i=1}^M C(W_i) \otimes M_{m_i}, \; \; f \to (\mathrm{Ad}(\nu_i)^{-1}\circ \gamma_i(f))_i,
\end{equation*}
where now we are viewing $P_i(C(W_i) \otimes M_{m_i})P_i$ as sitting inside $C(W_i) \otimes M_{m_i}$. After repeating the above proof to all other $j$'s, we may finally define
\begin{equation*}
    \varphi: \bigoplus\limits_{j=1}^N C(Z_j) \otimes M_{n_j} \rightarrow \bigoplus\limits_{i=1}^M C(W_i) \otimes M_{m_i}, \; \; (f_1,\ldots,f_j) \to \sum\limits_{j=1}^N \varphi_j(f_j).
\end{equation*}
It is straightforward to see that $\varphi$ is a unital homomorphism. If $\varphi(f_1,\ldots,f_N)=0$ then $\varphi_j(f_j)=0$ for all $j$. Fixing $j$ as before, this means that $\gamma_i(f_j)=0$ for all $i$,  which means that $\overline{\gamma}_i(f_j)=0$ for all $i$, meaning that $f_j(\omega^i_s(t))=0$ for all $i$ and $s$, which means that $f_j=0$ by property \ref{oneone}. Hence $\varphi$ is injective. Showing that $\varphi$ is maximally homogeneous reduces to showing that $\overline{\gamma}_i$ is maximally homogeneous, which follows by \ref{threethree}. Let $\mathcal{F} = F_1 \times F_2 \times \ldots \times F_N \subset  G_1 \times G_2 \times \ldots \times G_N \subset \bigoplus\limits_{j=1}^N C(Z_j) \otimes M_{n_j}$. For $a \in \mathcal{F}$, we have that $\norm{\phi(a)-\varphi(a)}$ is smaller than 
\begin{equation*}
\begin{split}
   & N\max\limits_{j,a_j \in G_j}\norm{\phi(\eta_j(a_j))-\varphi_j(a_j)} \le N \max\limits_i\max\limits_{j,a_j \in G_j}\norm{\phi_i(a_j)-\mathrm{Ad}(\nu_i)^{-1}(\gamma_i(a_j))} \\ = \ & N \max\limits_i\max\limits_{j,a_j \in G_j}\norm{\mathrm{Ad}(\nu_i)(\phi_i(a_j))-\gamma_i(a_j)} \\ \le \ & N \max\limits_i\max\limits_{j,a_j \in G_j}(\norm{\mathrm{Ad}(\nu_i)(\phi_i(a_j))-\chi_i(a_j)} + \norm{\chi_i(a_j)-\gamma_i(a_j)} ) \xrightarrow{\delta\to 0} 0
\end{split}
\end{equation*}
The final estimate is due to \eqref{delta 2} and property \ref{twotwo}. 

Recall the notation in the statement of the theorem. We have now shown that we may get from $\phi_n$ an injective and maximally homogeneous unital homomorphism $\varphi_n$. Furthermore, assuming $\mathcal{F}_n$ is a finite set generating $A_n$ such that $\mathcal{F}_n$ contains $$\big( \bigcup\limits_{k=1}^{n-1} \phi_{n,k}(\mathcal{F}_k) \big) \cup \big( \bigcup\limits_{k=1}^{n-1} \psi_{n,k}(\mathcal{F}_k) \big)$$ (with $\phi_{n,k}$ and $\psi_{n,k}$ as in Lemma \ref{lemma: sufficient iso limit}) we have shown that we may assume $\norm{\phi_n(a)-\varphi_n(a)} < 2^{-n}$ for all $a \in \mathcal{F}_n$. Then by using Lemma \ref{lemma: sufficient iso limit}, we obtain that $A \cong \ilim_n \gekl{A_n, \varphi_n}$ as desired.
 \end{proof}
\pars

\btheo
\label{thm:dim<=1}
Suppose that $A = \ilim_n \gekl{A_n,\phi_n}$ is a unital AH-algebra, where $A_n = \bigoplus_i C(Z_n^i) \otimes M_{r_n^i}$ and where each $Z_n^i$ is a Hausdorff continuum having covering dimension at most one. Then $A$ has a C*-diagonal. 
\etheo
\nopar

\begin{proof}
As explained in \cite[\S~3.1]{Ror}, we can write $A$ as another inductive limit $A \cong \ilim_n \gekl{\dot{A}_n,\dot{\phi}_n}$ such that $\dot{A}_n = \bigoplus_i C(\dot{Z}_n^i) \otimes M_{r_n^i}$,  where the $\dot{Z}_n^i$ are connected CW-complexes of covering dimension at most one. In dimension at most one, there is no difference between CW-complexes and simplicial complexes. Hence \cite[Theorem~2.1,~Remark~2.2]{EGL05} yields yet another description $A \cong \ilim_n \gekl{\ddot{A}_n,\ddot{\phi}_n}$, where $\ddot{A}_n = \bigoplus_i C(\ddot{Z}_n^i) \otimes M_{r_n^i}$ where the $\ddot{Z}_n^i$ are connected CW-complexes of covering dimension at most one, and where the connecting maps are now injective and unital.

Consider the first building block and first connecting map of this new system, and relabel it as $$\phi:\bigoplus_j C({Z}_j) \otimes M_{n_j}\rightarrow \bigoplus_i C({W}_i) \otimes M_{m_i}. $$
Let $C=\bigoplus_j C_j$ denote any C$^*$-diagonal of the domain. Let the restriction of $\phi$ from the $j^{\text{th}}$ summand of the domain to the $i^{\text{th}}$ summand of the codomain be labeled $\phi_{ij}$. Let $K_{ij}$ denote the rank of $\phi_{ij}(1)(w)$, for any $w \in W$ ($W$ is connected and so any choice will yield the same rank).  
\pars

Let $F^Z_0$ denote the finite set of those $j$ for which $Z_j$ is zero-dimensional, and $F^Z_1$ denote the finite set of those $j$ for which $Z_j$ is one-dimensional. Let $F^W_0$ denote the finite set of those $i$ for which $W_i$ is zero-dimensional, and $F^W_1$ denote the finite set of those $i$ for which $W_i$ is one-dimensional.  We have already seen in the proof of Theorem \ref{thm:dim=1} how we can consider a unital $*$-homomorphism $\overline{\phi_{ij}}: C(Z_j)\otimes M_{n_j} \rightarrow C(W_i)\otimes M_{n_jK_{ij}}$. In the case that $j\in F^Z_1$ and $i\in F^W_1$ we can then perturb $\overline{\phi_{ij}}$ to a maximally homogeneous connecting map $\gamma_{ij}$, with the property that if $\Gamma_{ij}$ denotes the subset of $Z_j$ which is the images of the eigenvalue functions associated to $\gamma_{ij}$, then $Z_j\setminus F_j=\bigcup\limits_{i\in F^W_1}\Gamma_{ij}$, where $F_j$ denotes the finite subset of $Z_j$ that is the images of the eigenvalue functions associated to $\phi_{ij}$ for all $i\in F^W_0$ (this is shown in the proof of Theorem \ref{thm:dim=1}). For $j\in F^Z_1$ and $i\in F^W_1$ we can then find by Proposition \ref{prop:MaxHomDiag} a unique C$^*$-diagonal $D_{ij}\subset C(W_i) \otimes M_{n_jK_{ij}}$ with the properties given in that proposition. We have that $\mathrm{dim}(D_{ij}(w))=n_jK_{ij}$ for all $w\in W_i$. Note that, as in the proof of Theorem \ref{thm:dim=1} we will view $C(W_i)\otimes M_{n_jK_{ij}}$ as sitting inside $C(W_i)\otimes M_{m_i}$, and hence we will abuse notation and view $D_{ij}$ as sitting inside $C(W_i)\otimes M_{m_i}.$ 

If $j\in F^Z_0$, then let $\{e^j_{kl}: k=1,\ldots,n_j\}$ be a set of matrix units in $M_{n_j}$ such that $C_j$ is diagonal with respect to these matrix units. Note that $\overline{\phi_{ij}}(e^j_{kk})$ is a projection corresponding to a vector bundle over $W_i$, but since complex vector bundles over finite graphs are trivial, we may write $\overline{\phi_{ij}}(e^j_{kk})$ as $\sum\limits_{s=1}^{K_{ij}}f^{i,j,k}_s$ where $\{f^{i,j,k}_s: s=1,\ldots,K_{ij}\}$ are orthogonal Murray-von-Neumann equivalent minimal projections in $C(W_i)\otimes M_{n_jK_{ij}}$. By the pairwise orthogonality of $\{e^j_{kk}: k=1,\ldots,n_j\}$ we get pairwise orthogonal minimal equivalent projections $\{f^{i,j,k}_s: k=1,\ldots,n_j; ~s=1,\ldots,K_{ij}\}$. Now declare $v^k_s$ to be the partial isometry implementing the equivalence between $f^{i,j,1}_1$ and $f^{i,j,k}_s$, so $(v^k_s)^*v^k_s=f^{i,j,1}_1$, $v^k_s(v^k_s)^*=f^{i,j,k}_s.$ Then declare $f^{kl}_{pq}=v^k_p(v^l_q)^*.$ Then it is easy to see that $\{f^{kl}_{pq}: 1\le k,l \le n_j; ~1\le p,q \le K_{ij}\}$ defines a system of matrix units in $C(W_i)\otimes M_{n_jK_{ij}}$ such that $f^{kl}_{pq}f^{{k^\prime} {l^\prime}}_{{p^\prime} {q^\prime}}=\delta_{l,k^\prime}\delta_{q,p^\prime}f^{k,l^\prime}_{p,q^\prime}$. Let $D_{ij}$ be the diagonal with respect to this system. Then $\mathrm{dim}(D_{ij}(w))=n_jK_{ij}$. Note also that by our construction of the matrix units above it becomes a straightforward calculation to see that $D_{ij}$ satisfies the conditions of Theorem \ref{thm:LimCartan} with respect to the map $\overline{\phi_{ij}}$, as normalizers of $M_{n_j}$ are those elements whose representation with respect to the matrix units $\{e^j_{kl}\}$ have at most one non-zero entry in every row or column, and this will be preserved by $\overline{\phi_{ij}}$. Similarly it is clear that $\overline{\phi_{ij}}$ commutes with the conditional expectations, as these are given by projection  down to the diagonal. 

If $j\in F^Z_1$ but $i\in F^W_0$ then note that the associated eigenvalue functions of $\overline{\phi_{ij}}$ have image in the finite set $F_j$. It is now easy to declare $\gamma_{ij}$ to be of the same form as $\overline{\phi_{ij}}$ but have eigenvalue functions all distinct but close to the original ones. Indeed let $G_j$ be a finite subset of $C(Z_j)\otimes M_{n_j}$ (which is to be chosen later). Whenever there are $k$ equal points $z\in (Z_j \cap F_j)$ we just perturb $k-1$ of them to be close to $z$ but such that $\norm{g(z)-g(*)}$ is made as small as we like for all $g \in G_j$ (here $*$ denotes a perturbed point). Hence we obtain a unital maximally homogeneous $*$-homomorphism $\gamma_{ij}$ and from Proposition \ref{prop:MaxHomDiag} a C$^*$-diagonal $D_{ij}$ in $C(W_i)\otimes M_{n_jK_{ij}}$. Again we view it as sitting inside $C(W_i) \otimes M_{m_i}$. 

Then, we declare $D_i$ to be $\sum\limits_j D_{ij} \cong \bigoplus\limits_j D_{ij}$ as the C$^*$-algebras $D_{ij}$ are orthogonal over $j$. Note that $D_i$ is a maximally homogeneous subalgebra of $C(W_i)\otimes M_{m_i}$ as $\mathrm{dim}(D_i(w))=\sum \limits_j n_jK_{ij}=m_i$ and hence a C$^*$-diagonal by \cite[Lemma~4.4.3]{Raa_PhD}. Then we declare $D$ to be $\bigoplus\limits_iD_i$ which is clearly a C$^*$-diagonal in $A_2$. 

The corresponding perturbed $*$-homomorphism is obtained by first treating $\gamma_{ij}$ as a map into $C(W_i)\otimes M_{m_i}$ as in the proof of Theorem \ref{thm:dim=1}, then $\gamma_i:=\sum\limits_j \gamma_{ij}$ and finally $\varphi=\bigoplus\limits_i \gamma_i$. The constructions above ensure that the eigenvalue functions associated to $\varphi$ have images covering all of $\bigsqcup\limits_j Z_j$ and hence $\varphi$ is injective. By construction $\varphi(C)\subset D$ with the properties required by Theorem \ref{thm:LimCartan}. We inductively repeat this procedure on the building blocks, and then by using Lemma \ref{lemma: sufficient iso limit} and choosing the finite sets $G_j$ above appropriately with regards to this lemma, we obtain that $A\cong\varinjlim\{A_n,\varphi_n\}$. The proof is finished by appealing to Theorem \ref{thm:LimCartan}.
\end{proof}
\pars

We record the following consequence of Theorem~\ref{thm:dim<=1} (or Theorem \ref{thm:dim=1}).
\bcor
\label{cor:AHbdddim}
Every unital AH-algebra with bounded dimension, the ideal property, and torsion-free $K$-theory has a C*-diagonal.
\ecor
\nopar

\bproof
By \cite{GJLP}, such an AH-algebra is an A$\Tz$-algebra. Now apply Theorem~\ref{thm:dim<=1}.
\eproof
\pars

\section{Cartan subalgebras in AH-algebras with generalized diagonal connecting maps}

\label{s:DiagConn}

We consider AH-algebras of the following form: For each $n$ and $i$, let $Z_n^i$ be a compact, connected, Hausdorff space. Let $r_n$ be some natural numbers. Let $p_n^i \in C(Z_n^i) \otimes M_{r_n}$ be projections. Set $A_n \defeq \bigoplus_i p_n^i ( C(Z_n^i) \otimes M_{r_n} ) p_n^i$. Let $Z_n \defeq \coprod_i Z_n^i$. Let $\cY(n)$ be an index set. For $y \in \cY(n)$, let $\lambda_y: \: Z_{n+1} \to Z_n$ be continuous maps. Let $q_y \in C(Z_{n+1}) \otimes M_{s_n}$, $y \in \cY(n)$, be pairwise orthogonal projections. 
\bdefin
\label{def:GenDiag}
A generalized diagonal homomorphism from $A_n$ to $A_{n+1}$ is of the form
\begin{equation}\label{general formX}
 \varphi_n: \: A_n \to A_{n+1}, \, a \mapsto \sum_{y \in \cY(n)} (a \circ \lambda_y) \otimes q_y.
\end{equation}
\edefin
\nopar

Here $(a \circ \lambda_y) \otimes q_y$ denotes the function sending $z \in Z_{n+1}$ to $a(\lambda_y(z)) \otimes q_y(z) \in M_{r_n} \otimes M_{s_n} \cong M_{r_{n+1}}$, where we are using a fixed isomorphism $M_{r_n} \otimes M_{s_n} \cong M_{r_{n+1}}$. In particular, this means that we must have $r_n s_n = r_{n+1}$.
\pars

The AH-algebra we are interested in is given by $A = \ilim_n \gekl{A_n, \varphi_n}$.

We would like unital homomorphisms, so that requires $\varphi_n(p_n) = p_{n+1}$, and thus $p_{n+1} = \sum_{y \in \cY(n)} (p_n \circ \lambda_y) \otimes q_y$. 
\pari

Moreover, we want injective homomorphisms, i.e., $\bigcup_{y \in \cY(n)} \lambda_y(Z_{n+1}) = Z_n$. 
\pars

For the construction of Cartan subalgebras and groupoid models, it is crucial to require that, over each connected component of $Z_n$, $p_n$ is a sum of line bundles, and that, again over each connected component $Z_{n+1}^j$ of $Z_{n+1}$, $q_y$ is zero or a line bundle, for each $y \in \cY(n)$. More precisely, we require that for all $y \in \cY(n)$, there exists an index $j(y)$ such that $q_y \vert_{Z_{n+1}^{j(y)}}$ is a line bundle and $q_y \vert_{Z_{n+1}^j} = 0$ for all $j \neq j(y)$. The relation $p_{n+1} = \sum_{y \in \cY(n)} (p_n \circ \lambda_y) \otimes q_y$ implies that if, over each connected component of $Z_n$, $p_n$ is a sum of line bundles, and if over each connected component of $Z_{n+1}$, $q_y$ is zero or a line bundle, then $p_{n+1}$ is a sum of line bundles over each connected component of $Z_{n+1}$.

Our setting is similar as the one in \cite{Niu}, except that in \cite{Niu}, the connecting maps are allowed to differ from the ones we consider up to an inner automorphism. This, however, does not change the inductive limit C*-algebra up to isomorphism. Moreover, as explained above, we need the additional requirements regarding $p_n$ and $q_y$. The setting in \cite[Example~3.1.6]{Ror} is a special case of our situation, where $p_n$ and $q_y$ are trivial.

Here are some more examples which fit naturally into our setting.
\nopar

\begin{itemize}
\item Villadsen algebras of the first kind \cite{Vil98}.
\item Tom's examples of non-classifiable C*-algebras \cite{Toms}.
\item Goodearl algebras \cite{Goo} (see also \cite[Example~3.1.7]{Ror}).
\item AH-algebras models for dynamical systems in \cite[Example~2.5]{Niu}
\item Villadsen algebras of the second kind \cite{Vil99}. Note that the identification $\alpha$ in \cite{Vil99} corresponds to our isomorphism $M_{r_n} \otimes M_{s_n} \cong M_{r_{n+1}}$.
\end{itemize}
\pars

\bremark
General AH-algebras might not admit inductive limit descriptions with generalized diagonal connecting maps. In particular, there might not exist eigenvalue functions $\lambda_y$ from $Z_{n+1}$ to $Z_n$ which are continuous.
\eremark

Now assume that we are given an AH-algebra $A = \ilim_n \gekl{A_n, \varphi_n}$ as above, with unital, injective, generalized diagonal connecting maps $\varphi_n$. Let us construct groupoid  models for $A$.

Recall the construction of Cartan subalgebras in inductive limits in \cite[Theorem~3.6]{BL17} and \cite[\S~5]{Li18} as in \S~\ref{ss:CartanLimC}.

Let $(A_n, B_n)$ be as in \cite[\S~6.1]{Li18}. Let $(G_n,\Sigma_n)$ be a twisted groupoid model for $(A_n, B_n)$ as in \cite[\S~6.1]{Li18}. More precisely, we apply \cite[Lemma~6.1 and Lemma~6.2]{Li18} to each direct summand $p_n^i ( C(Z_n^i) \otimes M_{r_n} ) p_n^i$ and then form the disjoint union over $i$. Here we are using our assumption that $p_n^i$ is a sum of line bundles.

Now we follow \cite[\S~6.1]{Li18} to construct groupoid models for the connecting maps. We set
$$
 A[\varphi_n] = \bigoplus_{y \in \cY(n)} ((p_n \circ \lambda_y) \otimes q_y) (C(Z_{n+1}) \otimes M_{r_{n+1}}) ((p_n \circ \lambda_y) \otimes q_y) \subseteq A_{n+1}
$$
and let $B[\varphi_n] \defeq \bigoplus_{y \in \cY(n)} B_{y,q}$, where $B_{y,q}$ is the Cartan subalgebra of 
$$
 A_{y,q} \defeq ((p_n \circ \lambda_y) \otimes q_y) (C(Z_{n+1}) \otimes M_{r_{n+1}}) ((p_n \circ \lambda_y) \otimes q_y)
$$
as constructed above, following \cite[\S~6.1]{Li18}. We have 
$$
 A_{y,q} \cong (p_n \circ \lambda_y)  (C(Z_{n+1}^{j(y)} \otimes M_{r_n}) (p_n \circ \lambda_y).
$$
Set
$$
 A_y \defeq (p_n \circ \lambda_y)  (C(Z_{n+1}^{j(y)} \otimes M_{r_n}) (p_n \circ \lambda_y),
$$
and let $B_y$ be the image of $B_{y,q}$ under the isomorphism $A_{y,q} \cong A_y$. Let $(G_y,\Sigma_y)$ be the twisted groupoid model for $(A_y,B_y)$ as in \cite[\S~6.1]{Li18}. Let $p_y: \: (G_y,\Sigma_y) \to (G_n,\Sigma_n)$ be the groupoid model for $A_n \to A_y, \, a \ma a \circ \lambda_y$. Now we define
$$
 (H_n,T_n) \defeq \coprod_{y \in \cY(n)} (G_y,\Sigma_y).
$$
We have $(H_n,T_n) \subseteq (G_{n+1},\Sigma_{n+1})$. Let $i_{n+1}$ denote the canonical inclusion. Moreover, construct a projection $p_n: \: (H_n,T_n) \onto (G_n,\Sigma_n)$ given by
$$
 (H_n,T_n) = \coprod_{y \in \cY(n)} (G_y,\Sigma_y) \overset{\coprod_y p_y}{\lori} (G_n,\Sigma_n).
$$
We obtain twisted groupoid morphisms
$$
 (G_n,\Sigma_n) \overset{p_n}{\leftarrow} (H_n,T_n) \to (G_{n+1},\Sigma_{n+1}),
$$
where the second map is the open embedding induced by the inclusion $A[\varphi_n] \into A_{n+1}$, which are groupoid models for the homomorphisms $A_n \to A[\varphi_n] \to A_{n+1}$, as in \cite[\S~6.2]{Li18}.

Define $(G_{n,0}, \Sigma_{n,0}) \defeq (G_n,\Sigma_n)$, $(G_{n,m+1}, \Sigma_{n,m+1}) \defeq p_{n+m}^{-1}(G_{n,m},\Sigma_{n,m})$ for all $n$ and $m = 0, 1, \dotsc$, $(\bar{G}_n,\bar{\Sigma}_n) \defeq \plim_m \gekl{(G_{n,m},\Sigma_{n,m}), p_{n+m}}$ for all $n$. Moreover, the inclusions $(H_n,T_n) \into (G_{n+1},\Sigma_{n+1})$ induce embeddings with open image $\bar{i}_n: \: (\bar{G}_n,\bar{\Sigma}_n) \into (\bar{G}_{n+1},\bar{\Sigma}_{n+1})$. Define $(\bar{G},\bar{\Sigma}) \defeq \ilim \gekl{(\bar{G}_n,\bar{\Sigma}_n), \bar{i}_n}$. As explained in \cite[\S~5]{Li18}, $(\bar{G},\bar{\Sigma})$ is a groupoid model for $(A,B)$ in the sense that we have a canonical isomorphism $A \isom C^*_r(\bar{G},\bar{\Sigma})$ sending $B$ to $C_0(\bar{G}^{(0)})$.

The following is now an immediate consequence of Theorem~\ref{thm:LimCartan} and Proposition~\ref{prop:HomGPDmodel} (see \cite{BL17} and \cite[\S~5]{Li18}).
\btheo
\label{thm:DiagDiagConn}
$B \defeq \ilim_n \gekl{B_n, \varphi_n}$ is a C*-diagonal of $A$, and we have $(A,B) \cong (C^*_r(\bar{G},\bar{\Sigma}),C_0(\bar{G}^{(0)})$.
\etheo

The following follows immediately.
\nopar

\bcor
\label{cor:DiagDiagConn_Ex}
The following classes of examples have C*-diagonals:
\begin{itemize}
\item Villadsen algebras of the first kind \cite{Vil98},
\item Tom's examples of non-classifiable C*-algebras \cite{Toms},
\item Goodearl algebras \cite{Goo} (see also \cite[Example~3.1.7]{Ror}), 
\item AH-algebras models for dynamical systems in \cite[Example~2.5]{Niu}, 
\item Villadsen algebras of the second kind \cite{Vil99}.
\end{itemize}
\ecor
\pars

\bcor
\label{cor:sr}
For each $m = 2, 3, \dotsc$ or $m = \infty$, there exists a twisted groupoid $(\bar{G},\bar{\Sigma})$ such that $C^*_r(\bar{G},\bar{\Sigma})$ is a unital, separable, simple C*-algebra of stable rank $m$, whereas $C^*_r(\bar{G})$ is a unital, separable, simple C*-algebra of stable rank one. 
\ecor
\nopar

\bproof
For $m = 2, 3, \dotsc$ or $m = \infty$, consider a Villadsen algebra $A = \ilim_n \gekl{A_n, \varphi_n}$ of the second type which is separable, simple and of stable rank $m$ as in \cite{Vil99}. Let $(\bar{G},\bar{\Sigma})$ be a twisted groupoid model for $A$, which exists by Corollary~\ref{cor:DiagDiagConn_Ex}. Using the groupoid models for building blocks and connecting maps from \cite[\S~6]{Li18}, it is straightforward to check that $C^*_r(\bar{G}) \cong \ilim_n \gekl{A_n,\ti{\varphi_n}}$, where $\ti{\varphi}_n(a) = \sum_{y \in \cY(n)} (a \circ \lambda_y) \otimes \ti{q}_y$, where $\ti{q}_y = 1_{Z_{n+1}^{j(y)}} \otimes e_y$, and $\menge{e_y}{y \in \cY(n), \, j(y) = y}$ are pairwise orthogonal standard rank one projections (with exactly one value $1$ on the diagonal and $0$ everywhere else). Also note that the construction in \cite{Vil99} starts with the projection $p_1$ corresponding to the trivial line bundle. Now it is an immediate consequence that $C^*_r(\bar{G})$ is a unital, separable, simple C*-algebra, and it follows from \cite[\S~4]{EHT} that $C^*_r(\bar{G})$ has stable rank one. 
\eproof
\pars

\bquestion
In the situation of Corollary~\ref{cor:DiagDiagConn_Ex}, we know that, in particular, $\bar{G}^{(0)} \cong \plim_n \gekl{G_n^{(0)}, p_n}$. Does it follow that for all groupoid models for non-classifiable C*-algebras from \cite{Vil99,Toms}, we have $\dim \bar{G}^{(0)} = \infty$?
\equestion

\end{document}